\documentclass[a4paper,leqno,11pt]{amsart}

\usepackage{amsfonts,amssymb,verbatim,amsmath,amsthm,latexsym,textcomp,amscd}
\usepackage{latexsym,amsfonts,amssymb,epsfig,verbatim,extarrows}
\usepackage{amsmath,amsthm,amssymb,latexsym,graphics,textcomp}
\usepackage{graphicx}
\usepackage{color}
\usepackage{url}
\usepackage{enumerate}
\usepackage[mathscr]{euscript}
\usepackage[all, cmtip]{xy} 
\usepackage{hyperref}
\hypersetup{
	colorlinks,
	citecolor=red,
	filecolor=black,
	linkcolor=blue,
	urlcolor=black
}

\theoremstyle{definition}
\newtheorem{theorem}{Theorem}[section]
\newtheorem{prop}[theorem]{Proposition}

\newtheorem{defn}[theorem]{Definition}

\newcommand{\Z}{{\mathbb{Z}}}

\newcommand{\Q}{{\mathbb Q}}

\newcommand\A{{\mathcal A}}

\newcommand\FF{{\mathcal F}}

\newcommand\LL{{\mathcal L}}
\newcommand\MM{{\mathcal M}}

\newcommand\PP{{\mathcal P}}

\newcommand\PMF{{\PP\kern-2pt\MM\FF}}

\newcommand\PML{{\PP\kern-2pt\MM\LL}}

\newcommand{\fsubd}{\mathrel{{\scriptstyle\searrow}\kern-1ex^d\kern0.5ex}}
\newcommand{\bsubd}{\mathrel{{\scriptstyle\swarrow}\kern-1.6ex^d\kern0.8ex}}
\newcommand{\fsubeq}{\mathrel{\raise-.7ex\hbox{$\overset{\searrow}{=}$}}}
\newcommand{\bsubeq}{\mathrel{\raise-.7ex\hbox{$\overset{\swarrow}{=}$}}}

\newcommand{\tsh}[1]{\left\{\kern-.9ex\left\{#1\right\}\kern-.9ex\right\}}

\newcommand{\lra}{\longrightarrow}

\newcommand{\Lie}{\mathit{Lie}}

\begin{document}

\title[]{Homotopy groups of certain highly connected manifolds via loop space homology}

\author{Samik Basu}
\email{samik.basu2@gmail.com; samik@rkmvu.ac.in}
\address{Department of Mathematics,
Vivekananda University, Belur, Howrah, 
West Bengal, India.}

\author{Somnath Basu}
\email{basu.somnath@gmail.com; somnath@rkmvu.ac.in}
\address{Department of Mathematics,
Vivekananda University, Belur, Howrah,
West Bengal, India.}
 
%\date{\today}



\maketitle

\begin{abstract}
For $n\geq 2$ we consider $(n-1)$-connected closed manifolds of dimension at most $(3n-2)$. We prove that away from a finite set of primes, the $p$-local homotopy groups of $M$ are determined by the dimension of the space of indecomposable elements in the cohomology ring $H^\ast(M)$. Moreover, we show that these $p$-local homotopy groups can be expressed as direct sum of $p$-local homotopy groups of spheres. This generalizes some of the results of our earlier work \cite{BaBa15_unpub}.
\end{abstract}

%\tableofcontents

%\setcounter{section}{0}

\section{Introduction}\label{intro}

In this document we consider $(n-1)$-connected closed manifolds of dimension at most $(3n-2)$ and prove analogous results to those for $(n-1)$-connected $2n$-manifolds in \cite{BaBa15_unpub}. We shall prove the following result (cf. Theorem \ref{htpy}, Theorem \ref{htpyform}).
\begin{theorem}
{\it Let $M$ be a closed $(n-1)$-connected $d$-manifold with $n\geq 2$, $d\leq 3n-2$ and $\dim H^\ast (M ) > 4$. Let $r$ denote the dimension of the space of indecomposables in $H^\ast(M;\Q)$. Then there is a finite set of primes $\Pi$ such that for $p\notin \Pi$, \\
\textup{(a)} the $p$-local homotopy groups of $M$ are determined by $r$;\\
\textup{(b)} the $p$-local homotopy groups of $M$ can be expressed as a direct sum of $p$-local homotopy groups of spheres.   
}
\end{theorem}

If a generator of $H^d(M;\Q)$ is indecomposable then it follows from Poincar\'{e} duality that the rational cohomology of $M$ is that of $S^n$. In this case $\dim H^\ast(M;\Q)=2$ and $r=1$. Conversely, suppose that $r=1$ and $M$ does not have the cohomology of a sphere. It follows that $M$ is a manifold of dimension $2n$ with $H^\ast(M;\Q)=\Q[x]/(x^3)$ and $n$ is even. 

By the assumptions on $M$, the cup product of any three cohomology classes (of positive degree) is zero. Now we assume that $\dim H^\ast(M;\Q)>4$. If any class $\alpha=a\cup b\in H^i(M;\Q)$ (with $i<d$) is reducible then by Poincar\'{e} duality there exists $\beta\in H^{d-i}(M;\Q)$ such that $\alpha\cup\beta$ is a generator of $H^d(M;\Q)$. Thus, $a\cup b\cup \beta\neq 0$ and this violates our previous observation.  Therefore, if $r>1$ we deduce that 
$$
r= \dim( \oplus_{0<i < d} H^i(M;\Q))=\sum_{0<i<d}\dim H^i(M;\Q)=\dim H^\ast(M;\Q)-2.
$$
We note that the condition $\dim H^\ast (M ) > 4$ is equivalent to $r \geq 3$.  In terms of rational homotopy groups, $M$ is rationally hyperbolic if and only if $r>2$. In this case we have the following result (cf. Theorem \ref{Moorehcm}).

\begin{theorem}
{\it Let $M$ be a closed $(n-1)$-connected $d$-manifold with $n\geq 2$, $d\leq 3n-2$ and $\dim H^\ast(M;\Q)>4$. Then the homotopy groups of $M$  has unbounded $p$-exponents for all but finitely many primes.}
\end{theorem}

The above result verifies the Moore conjecture (see the discussion before Theorem \ref{Moorehcm} as well as \cite{FHT01} pp. 518) for such spaces in the rationally hyperbolic case. The low rank cases, i.e., when $r=1,2$ are discussed in \S \ref{r12} (see Theorem \ref{lowrank}).\\

\noindent
\thanks{The authors would like to thank Alexander Berglund for certain helpful discussions.}

\section{Homology of the loop space}\label{homloop}

In \cite{BerBor15}, the homology of $\Omega M$ is computed for $M$ a $(n-1)$-connected $(3n-2)$ manifold. Let us recall it. Let $n \geq 2$ and suppose that $M$ is an $(n - 1)$-connected closed manifold of dimension $d \leq 3n -2$ such that $\dim H^\ast (M ) > 4$. Choose a basis $x_1 , \cdots , x_r$ for the indecomposables of $H^\ast  (M)$. If we choose an orientation class $[M]$ for $M$ then let $c_ {ij} = \langle x_ i x_j , [M ]\rangle$. Consider the homology ring $H_\ast(\Omega M)$ of the based loop space, equipped with the Pontrjagin product. This ring is freely generated as an associative algebra by classes $u_1 , \cdots, u_r $ whose homology suspensions are dual to the classes $x_1 , \cdots, x_r$ (in particular $|u_i | = |x_i | - 1$), modulo the single quadratic relation
$$\sum_{i,j}(-1)^{ |u_i |} c_ {ji} u_i u_j = 0.$$

Let $M$ be a closed $(n-1)$-connected $d$-manifold with $d\leq 3n-2$. The cohomology of $M$ is finitely generated and has $p$-torsion only for a finite set of primes $p$. Let $\Sigma$ be the set of primes such that the cohomology of $M$ has $p$-torsion. Define 
$$R_\Sigma= \Z[\textstyle{\frac{1}{p}}\,| \,p\in \Sigma].$$
Then we may deduce the following facts.\\
\hspace*{0.5cm} (a) $H^*(M;R_\Sigma)$ is a free $R_\Sigma$-module. \\
\hspace*{0.5cm} (b) The natural map $H^*(M;R_\Sigma) \otimes_{R_\Sigma} \Q \to H^*(M;\Q)$ is an isomorphism. \\
The first fact follows from Universal Coefficient Theorem for cohomology and the defining property of $R_\Sigma$. The second fact is clear.\\

As noted earlier, the only non-trivial products of positive dimensional classes are given by the intersection form. Therefore the module of indecomposables $\A(M)= \oplus_{0< i < d} H^i(M;R_\Sigma)$. Let $x_1,\cdots,x_r$ be a basis of $\A(M)$. Fix a choice of an orientation class $[M]\in H^d(M;R_\Sigma)$ of $M$. Let $c_ {ij} = \langle x_ i x_j , [M ]\rangle$. Let $u_1 , \cdots, u_r $ denote classes whose homology suspensions are dual to the classes $x_1 , \cdots, x_r$ (in particular $|u_i | = |x_i| - 1$). With coefficients in a field $k=\Q$ or a quotient field of $R_\Sigma$ we have the following result for the homology of the loop space with respect to the Pontrjagin product (cf. \cite{BerBor15}, Theorem 1.1).
\begin{prop}\label{loopk}
As associative rings, 
$$H_*(\Omega M;k) \cong T_k(u_1,\cdots,u_r)/\big(\textstyle {\sum} (-1)^{ |u_i |+1} c_ {ji} u_i u_j \big)$$ 
\end{prop}
\noindent This directly leads us to the following integral version.
\begin{prop}\label{manhomloop}
As associative rings,
 $$H_*(\Omega M;R_\Sigma) \cong T_{R_\Sigma}(u_1,\cdots,u_r)/\big(\textstyle {\sum} (-1)^{ |u_i |+1} c_ {ji} u_i u_j \big)$$ 
\end{prop}  
\begin{proof}
Since $M$ is an orientable manifold the homology $H_*(M-pt)$ matches $H_\ast M$ in all degrees upto $(d-1)$. From the conditions on $M$ we deduce that $H^*(M-pt)$ is free on the classes $x_i^\ast$ which are dual to the classes $x_i$ and the products are all zero. It follows that $H_*(\Omega (M-pt))$ is a tensor algebra on the classes $u_i$. Therefore we have a map 
$$\phi: T_{R_\Sigma}(u_1,\cdots,u_r) \stackrel{\cong}{\lra} H_*(\Omega(M-pt))\lra H_*(\Omega M).$$

In the dimension range $0\leq * \leq d-1$, we may compute $H_*(\Omega M)$ using the Serre spectral sequence associated to the path-space fibration $\Omega M \to PM \to M$. This has the form 
$$E^2_{p,q} = H_p(M) \otimes H_q(\Omega M) \implies H_*(\mathit{pt})$$
with coefficients in $R_\Sigma$. From the multiplicative structure on the dual cohomology spectral sequence it follows that  the indecomposable elements (with basis $x_j^\ast$) lie in the image of the transgression. Therefore the classes $x_j$ are transgressive and the transgress onto the classes $u_j$. Thus in the spectral sequence we have $d(x_j)=u_j$. 

The homology of $M$ being torsion-free implies that the cohomology of $M$ is just the dual. From the dual spectral sequence, we deduce that the classes $x_k\otimes u_j$ are mapped by differentials onto the classes $u_k\otimes u_j$ on the vertical $0$-line. It follows that in degrees $\leq d-1$, $H_*(\Omega M)$ are generated by the classes $u_i, u_iu_j$. The differential on the class $[M]$ hits a linear combination of $x_k \otimes u_j$. Hence, in this range of degrees $H_*(\Omega M) \cong T_{R_\Sigma}(u_1,\cdots,u_r) /(l)$ for some element $l$ of homogeneous degree $2$ in $u_i$. 

Let $H_*(\Omega M)^{(2)}$ denote the free $R_\Sigma$-submodule generated by the homogeneous degree $2$ elements which is isomorphic to $R_\Sigma\{u_i\otimes u_j\}/(l)$. The computations of \cite{BerBor15} as quoted above imply that 
$$R_\Sigma\{u_i\otimes u_j\}/(l) \otimes_{R_\Sigma} k \cong R_\Sigma\{u_i\otimes u_j\}/\big(\textstyle{\sum} (-1)^{ |u_i |+1} c_ {ji} u_i u_j \big) \otimes_{R_\Sigma} k$$
for $k$ being either the fraction field of $R_\Sigma$, or $R/(\pi)$ for primes $\pi$ in $R_\Sigma$. The first case implies that there are $a,b \in R_\Sigma$ such that 
$$al = b \big(\textstyle{\sum} (-1)^{ |u_i |+1} c_ {ji} u_i u_j\big) $$
and the second cases imply that $a$ and $b$ are non-zero and differ by a unit modulo $\pi$ for every prime $\pi$. Thus $a$ and $b$ are forced to be units after possible cancellations, and we may take $l=\sum (-1)^{ |u_i |+1} c_ {ji} u_i u_j $. Thus
$$H_\ast(\Omega M)^{(2)} \cong R_\Sigma \{u_i \otimes u_j \}/\big(\textstyle{\sum} (-1)^{ |u_i |+1} c_ {ji} u_i u_j \big)$$
so that the element $\sum (-1)^{ |u_i |+1} c_ {ji} u_i u_j  \in  T_{R_\Sigma}(u_1,\cdots,u_r)$ goes to $0$ under $\phi$ above. Thus we obtain a ring map 
$$ T_{R_\Sigma}(u_1,\cdots,u_r)/\big(\textstyle{\sum} (-1)^{ |u_i |+1} c_ {ji} u_i u_j ) \to H_*(\Omega M;R_\Sigma\big) $$
which is an isomorphism after tensoring with the fraction field of $R_\Sigma$ or going modulo a prime from Proposition \ref{loopk}. The result now  follows.
\end{proof}

\section{Homotopy groups of certain $(n-1)$-connected manifolds}

In this section we deduce results about the homotopy groups of $(n-1)$-connected manifolds of dimension $d\leq 3n-2$ after inverting finitely many primes. We use the computation of the homology of the loop space in Section \ref{homloop}. Note from Proposition \ref{manhomloop} that $H_*(\Omega M)$ is a quadratic algebra. We prove that this possesses a nice basis and so does the corresponding quadratic Lie algebra. The basis of the Lie algebra is used to express $\pi_*(M)$ as a direct sum of homotopy groups of spheres after inverting finitely many primes.

\subsection{Algebraic preliminaries}
We start by recalling some algebraic preliminaries on quadratic algebras and quadratic Lie algebras. For further details we refer to \cite{BaBa15_unpub,PolPos05,Nei10}. Let $A$ be a commutative ring (usually a principal ideal domain (PID)). If $V$ is a free $A$-module then we shall denote by $T_A(V)$ (often abbreviated as $T(V)$) the tensor algebra generated by $V$. The notation $\Lie(V)$ (respectively $\Lie^{gr}(V)$) denotes the free Lie algebra (respectively graded Lie algebra) on the $A$-module $V$.

\begin{defn}
For $R\subset V\otimes_A V$, the associative algebra $A(V,R)=T(V)/(R)$ is called a quadratic $A$-algebra.\\
\hspace*{0.5cm} If $R\subset V\otimes_A V$ lies in $\Lie(V)$, the Lie algebra $L(V,R)= \Lie(V)/((R))$ is called a quadratic Lie algebra over $A$. In the graded case this is denoted $L^{gr}(V,R)$. 
\end{defn} 

It may be observed that the universal enveloping algebra of $L(V,R)$ is $A(V,R)$ and in the graded case the universal enveloping algebra of $L^{gr}(V,R)$ is $A(V,R)$ as graded modules. If in addition the modules $A(V,R)$ and $L(V,R)$ are free, there is a Poincar\'e-Birkhoff-Witt theorem which may be stated as $E_0(A(V,R))\cong A[L(V,R)]$. The notation $A[L(V,R)]$ denotes the polynomial $A$-algebra on the module $L(V,R)$ and $E_0(A(V,R))$ denotes the associated graded for the filtration of $A(V,R)$ induced by the weight filtration on the tensor algebra. A similar statement holds for the graded case where one interprets the polynomial algebra as the polynomial algebra on even degree classes tensored with the exterior algebra on the odd degree classes. Finally from \cite{Car58}, \cite{Laz54} one may deduce that for a PID $A$, $L(V,R)$ is a free module if $A(V,R)$ is free and $V$ has finite rank. 

Next we recall the Diamond lemma from \cite{Berg78}. Suppose that the free module $R$ can be given a basis where each element is of the form $W_i-f_i$ where $W_i$ is a monomial. Call a monomial $R$-indecomposable if it does not possess any submonomial which occurs as $W_i$ in the above chosen basis. The Diamond lemma states certain sufficient conditions under which the $R$-indecomposable monomials form a basis of $A(V,R)$. The following implication of the Diamond lemma suffices for this paper.
\begin{prop}\label{Diamond}
Suppose that $R$ is generated by a single element of the form 
$$x_a\otimes x_\beta = \sum_{(i,j)\neq (\alpha, \beta)} a_{i,j}\,\, x_i \otimes x_j$$ 
with $\alpha\neq \beta$. Then the $R$-indecomposable elements form a basis for $A(V,R)$. 
\end{prop}  
There is an analogous construction for Lie algebras $L(V,R)$ defined by generators and relations (see \cite{LaRam95}). This is called a Lyndon basis. Start with a basis of $V$ and an order on the basis set. We call a word in elements of $V$ a Lyndon word if it is lexicographically smaller than its cyclic rearrangements. For a Lyndon word $l$ there are unique Lyndon words $l_1$ and $l_2$ so that $l=l_1l_2$ and $l_2$ is the largest possible Lyndon word occuring in the right in $l$. Inductively we may associate to the Lyndon word $l$ the free Lie algebra element $b(l)=[b(l_1),b(l_2)]$. One verifies that this is a basis of the free Lie algebra on $V$. In the case $R\neq 0$, we say a Lyndon word is $R$-standard if it cannot be further reduced using relations in $R$ (with respect to the chosen order on $V$). From \cite{LaRam95} we recall the following result (also see \cite{BaBa15_unpub}, Theorem 2.17).
\begin{prop}\label{liebasis}
Suppose that $A$ is a localization of $\Z$ and $R$ as in Proposition \ref{Diamond}. Then, the $R$-standard Lyndon words give a basis of $L(V,R)$.   
\end{prop} 

\subsection{Homotopy groups using loop space homology}
Let us denote by $l(M)$ the sum $\sum (-1)^{ |u_i |+1} c_ {ji} u_i u_j$. It is clear that $l(M)$ lies in the free graded Lie algebra on the classes $u_1,\cdots,u_r$ which we denote by  $\Lie^{gr}(u_1,\cdots,u_r)$.  Consider the graded Lie Algebra $\LL_k^{gr}(M)$ (over $\Z$) given by 
$$\frac{\Lie^{gr}(u_1,\cdots,u_k)}{(l(M))}$$
where $(l(M))$ denotes the graded Lie algebra ideal generated by $l(M)$. This is a quadratic graded Lie algebra. In this respect we denote 
$$l(M)= \sum_{i<j} l_{i,j}[u_i,u_j]^{gr} =\sum_{i<j} l_{i,j}(u_i\otimes u_j - (-1)^{|u_i||u_j|}u_j\otimes u_i).$$
We make an analogous ungraded construction. Consider the element 
$$l^u(M)= \sum_{i<j} l_{i,j} [u_i,u_j] = \sum_{i<j} l_{i,j} (u_i\otimes u_j - u_j\otimes u_i).$$
This element lies in $\Lie(u_1,\ldots,u_r) $. We shall make use of the following notation: 
$$A_r^u(M)= \frac{T_{R_\Sigma}(u_1,\ldots,u_r)}{(l^u(M))},\,\,\LL_r^u(M)= \frac{\Lie(u_1,\ldots,u_r)}{(l^u(M))}.$$ 
The Lie algebra $\LL_r^u(M)$ and the associative algebra $A_r^u(M)$ possess an induced grading. 
 
 \begin{prop}\label{lynM}
 $\LL^u_r(M)$ and $\LL^{gr}_r(M)$ are free over $R_\Sigma$. The Lyndon basis gives a basis of $\LL^u_r(M)$. 
 \end{prop}

 \begin{proof}
 From the formulas in Proposition \ref{manhomloop}, it is clear that $H_*(\Omega M)$ is the universal enveloping algebra of the Lie algebra $\LL^{gr}(M)$ in the graded sense. Analogously, $A_r^u(M)$ is the universal enveloping algebra of $\LL^u(M)$. As $R_\Sigma$ is a PID, for the first statement it suffices to show that $H_*(\Omega M)$ and $A_r^u(M)$ are free $R_\Sigma$-module. We verify this last fact by proving $l(M)$ and $l^u(M)$ satisfies the hypothesis of Proposition \ref{Diamond}. Now Proposition \ref{liebasis} implies the second statement as well. 

Since the coefficients of $l(M)$ and $l^u(M)$ differ only by a sign, it suffices to write the element $l(M)$ as 
$$u_iu_j = \mathit{terms~not~containing~}u_iu_j.$$ 
This is equivalent to a change of basis of the $x_i$ so that some $c_{ji}$ equals $1$. We have seen that this is possible for $d=2n$ in \cite{BaBa15_unpub}. For $d>2n$ note that $n$ and $d-n$ are not the same. Now pick the basis $x_i$ so that the dual classes in $H^n(M)$ are Poincar\'e dual to those in $H^{d-n}(M)$ (this is possible as $n\neq \frac{d}{2}$). For example we may start with a basis of $H^n(M)$ and then the dual basis of $H^{d-n}(M)$ and extend to a basis of $H^{0<*<d}(M)$. Order the basis so that $x_1^* \in H^n(M)$ and $x_r^*\in H^{d-n}(M)$ is the dual element. Then $c_{1,r}=1$ and thus 
$$l(M)= \sum (-1)^{ |u_i |+1} c_ {ji} u_i u_j = u_1 u_r + \mathit{combination~of~other~terms}$$
This completes the proof.
 \end{proof}
 
Next we enlarge the set of primes so that the classes $u_i$ are in the image of the Hurewicz homomorphism. 
\begin{prop}\label{hur}
There exists a finite set of primes $\Gamma$ containing $\Sigma$ such that the classes $u_i$ lie in the image of the Hurewiciz homomorphism $\pi_*(\Omega M) \otimes R_\Gamma \to H_*(\Omega M; R_\Gamma)$. 
\end{prop} 

\begin{proof}
Consider the commutative diagram 
$$\xymatrix{\pi_*(\Omega M) \otimes R_\Sigma \ar[r]^{Hur} \ar[d] & H_*(\Omega M; R_\Sigma) \ar[d] \\
                      \pi_*(\Omega M) \otimes \Q \ar[r]^{Hur}                     & H_*(\Omega M; \Q)        }. $$
Since $H_\ast(\Omega M;R_\Sigma)$ is a free $R_\Sigma$-module, the right vertical arrow is injective; it takes $u_i$ to the corresponding element $u_i$.
We know from the Milnor-Moore theorem that $H_\ast(\Omega M;\Q)$ is the universal enveloping algebra on the rational homotopy Lie algebra $\pi_\ast(M)\otimes\Q$. It follows by standard methods (for e.g., from formality of $M$ and minimal models) that $x_i$'s are the generators of homotopy Lie algebra in the appropriate degrees. Thus, the element $u_i$ lie in the Hurewicz homomorphism for $\Q$ coefficients. It follows that for every $i$ there is an integer $d_i$ so that $d_i u_i$ lies in the image of the Hurewicz homorphism. Define $\Gamma$ as $\Sigma$ plus all the prime factors of $d_i$ for $1\leq i \leq r$. Consequently, over $R_\Gamma$, all the $u_i$ are in the image of the Hurewicz homomorphism.
\end{proof}

Our goal is to compute the homotopy groups $\pi_*(M)\otimes R_\Gamma$. We work in the {\it $R_\Gamma$-local category}: that is, the category obtained from spaces by localizing with respect to $H_*( - ; R_\Gamma)$-equivalences. Let $S^n_\Gamma$ denote the $R_\Gamma$-local sphere. We know that if a map between simply connected spaces (or, more generally, simple spaces) is a $H_*(-;R_\Gamma)$-equivalence then it induces an isomorphism on $\pi_\ast(-)\otimes R_\Gamma$.  

From Proposition \ref{hur}, there are elements in $\pi_\ast(M)\otimes R_\Gamma$ which loops down to $u_i$ under the Hurewicz map. By iterated Whitehead products we may map spheres into $M$ corresponding to chosen elements of the Lie algebra $\LL^u(M)$. We describe this in a precise fashion below. 

Let $d(u_i)$ denote the degree of $u_i$. We fix a map $S^{d(u_i)} \to \Omega M$ which maps to $u_i$ under the Hurewicz homomorphism. By adjunction we have a map $\alpha_i: S^{d(u_i)+1} \to M$ with the property that after looping $\alpha_i$ the generator of the Pontrjagin ring maps to $u_i$. 

There exists a Lyndon basis for $\LL^u(M)$ by Proposition \ref{lynM}. Enlist these elements in order as $l_1 < l_2 <\ldots$ and define the height of a basis element by $h_i= h(l_i)=k+1$ if $b(l_i) \in (\LL^u(M))_k$ the $k^{th}$-graded piece. Then $h(l_i)\leq h(l_{i+1})$. Note that $b(l_i)$ represents an element of $\Lie(u_1,\ldots,u_r)$ and is thus represented by an iterated Lie bracket of $u_i$.  Define $\lambda_i : S^{h_i}_\Gamma \to M$ as the Whitehead product replacing each $u_i$ in the bracket by $\alpha_i$.  

%Next we formulate a basis of $\LL_k$ following \cite{Hil55}. We inductively define a basis of $\LL_k(r)$ and also an order on the basis. The degree one basic elements are defined to be $\alpha_1,\ldots,\alpha_k, \beta_1,\ldots,\beta_k$ and the order is $\alpha_1<\ldots < \alpha_k<\beta_1<\ldots <\beta_k$. The basic elements of degree $2$ are 
%$$[\alpha_i,\alpha_j],~[\beta_i,\beta_j]~\mathit{for}~i\neq j,~[\alpha_i,\beta_j]~\mathit{for}~(i,j)\neq (k,k)$$
%and the fix an order arbitrarily among them and in addition declare that a basic element of degree 2 is greater than those of degree 1.  Now suppose basic elements of degree $<w$ are defined and ordered. A basic element of weight $w$ is defined by $[a,b]$ where $a$ is a basic element of degree $u$ and $b$ is a basic element of degree $v$ with $u+v=w$ such that $a<b$ and if $b$ is a bracket $[c,d]$ then $c\leq a$.  

%\begin{prop}
%The set basic elements of degree $r$ form a basis of the free abelian group $\LL_k(r)$. 
%\end{prop}

\begin{theorem}\label{htpy}
There is an isomorphism
$$\pi_*(M) \otimes R_\Gamma \cong \sum_{i\geq 1} \pi_* S^{h_i}\otimes R_\Gamma $$ 
and the inclusion of each summand is given by $\lambda_i$. 
\end{theorem}
\noindent Observe that the right hand side is a finite sum in each degree.
\begin{proof}
The maps $\Omega \lambda_i: \Omega S^{h_i}_\Gamma \rightarrow \Omega M$ for $i=1,\ldots,n$ can be multiplied using the H-space structure on $\Omega M$ to obtain a map from $S(n)=  \prod_{i=1}^n \Omega S^{h_i}_\Gamma \rightarrow \Omega T$. Letting $n$ vary $S(n)$ gives a directed system arising from the inclusion of subfactors using the basepoint. Fix an associative model for $\Omega T$ (for example using Moore loops) and observe that the various maps from $S(n)$ induces a map on the homotopy colimit
$$\Lambda: S := \mathit{hocolim}_n~ S(n) \longrightarrow \Omega M.$$   
Note that homotopy groups of $S$ is the right hand side of the expression in the Theorem shifted in degree by $1$. Hence it suffices to prove that $\Lambda$ is a weak equivalence after inverting the primes in $\Gamma$. As both the domain and codomain are simple spaces, it suffices to show that this is a $R_\Gamma$-homology isomorphism. 

The homology of $S$ is a polynomial algebra with a generator for each copy of $\Omega S^{h_i}_\Gamma$ 
$$H_*(S) \cong T_{R_\Gamma} (c_{h_1 -1})\otimes T_{R_\Gamma} (c_{h_2-1})\ldots \cong R_{\Gamma}[c_{h_1-1},c_{h_2 -1},\ldots ]$$
and $\Lambda_*c_{h_i-1}$ is the  Hurewicz image of $\lambda_i \in H_{h_i-1}(\Omega M)$. Denote $\rho$ as  
$$\rho : \pi_n(X)\cong \pi_{n-1}(\Omega X) \xrightarrow{\mathit{Hur}} H_{n-1}(\Omega X)$$
We know from \cite{Hil55} that 
\begin{equation}\label{lhur}
\rho([a,b])=\rho(a)\rho(b) - (-1)^{|a||b|}\rho(b)\rho(a).
\end{equation}
 
Now from Theorem \ref{manhomloop} that $H_*\Omega M \cong T_{R_\Gamma}(u_1,\cdots,u_r)/(l(M)) $ the universal enveloping algebra of $\LL_r^{gr}(M)$ (in the graded sense). From the Poincar\'e-Birkhoff-Witt theorem for graded Lie algebras we have 
$$E_0 T_{R_\Gamma}(u_1,\cdots,u_r)/(l(M)) \cong E(\LL^{gr}_r(M)^{\mathit{odd}} ) \otimes P(\LL^{gr}_r(M)^{\mathit{even}} )$$
The map $\rho$ carries each $\alpha_i$  to $u_i$. The element  $b(l_i)$ is mapped inside $H_*(\Omega M)$ to the element corresponding to the graded Lie algebra element by equation (\ref{lhur}). We prove that $T(a_1,\ldots,a_r)/(l(M))$ has a basis given by monomials on $\rho(b(l_1)),\rho(b(l_2)),\ldots$.

Observe inductively that all the elements in $\LL^{gr}_r(M)$ can be expressed as linear combinations of monomials in $\rho(b(l_i))$.  It is clear for elements of weight $1$. For the weight $2$ elements note that they are generated by $[u_i,u_j]^{gr}$ for $i<j,~ (i,j)\neq (1,2)$ and $u_i^2$ if $d(u(i))$ is odd. The former are the Lyndon words and the latter is the square of a monomial. In the general case, a graded Lie algebra element is either a monomial or the square of a lower odd degree class; from one of the conditions in the definition of a graded Lie algebra the bracket with a square can be expressed as a bracket. Such a monomial may be obtained by applying $\rho$ on the corresponding ungraded element. This is a linear combination of certain $b(l_i)$ and something in the ideal generated by $l^u(M)$. Applying $\rho$ we obtain a combination of $\rho(b(l_i))$ and something in the ideal generated by $l(M)$ as $\rho(l^u(M))=l(M)$ which verifies the induction step. As an application of the Poincar\'e-Birkhoff-Witt Theorem, we know that $\Lambda_*$ is surjective.   

Hence we have that the graded map
$$ R_\Gamma[\rho(b(l_1)),\rho(b(l_2)),\ldots] \to     T_{R_\Gamma}(u_1,\ldots,u_r)/(l(M))$$
is surjective. We also know 
$$                      R_\Gamma [b(l_1),b(l_2),\ldots]    \to T_{R_\Gamma}(u_1,\ldots,u_r)/(l^u(M))$$
is an isomorphism. Now both $T_{R_\Gamma}(u_1,\ldots,u_r)/(l(M))$ and $T_{R_\Gamma}(u_1,\ldots,u_r)/(l^u(M))$ have bases given by the Diamond lemma and thus are of the same graded dimension. It follows that the graded pieces of $ R_\Gamma [\rho(b(l_1)),\rho(b(l_2)),\ldots]$ and $     T_{R_\Gamma} (u_1,\ldots,u_r)/(l(M))$ have the same rank which is finite. Thus on graded pieces one has a surjective map between free $R_\Gamma$-modules of the same rank which must be an isomorphism.
\end{proof}

We may now compute the number of copies on of $S^k$ in the expression of Theorem \ref{htpy} from the rational cohomology groups of $M$. Let 
$$q_M(t)= 1 - \sum_{n-1<i<d} b_i(M)t^i + t^d$$
Then $\frac{1}{q_M(t)}$ is the generating series for $\Omega M$ (see \cite{LoVa12}, Theorem 3.5.1) and the fact that $H_*(\Omega M)$ is Koszul as an associative algebra (cf. \cite{BerBor15}). Let
$$\eta_m := \mbox{coefficient of $t^m$ in }log(q_M(t)).$$
We may repeat the proof of Theorem 5.7 of \cite{BaBa15_unpub} to deduce the following result.
\begin{theorem}\label{htpyform}
 The number of groups $\pi_s S^j\otimes R_\Gamma$ in $\pi_s(M)\otimes R_\Gamma$ is 
$$l_{j-1}= - \sum_{d|j-1} \mu(d)\frac{\eta_{(j-1)/d}}{d}$$
where $\mu$ is the M\"obius function.
\end{theorem}

%\begin{proof}
%It is enough to compute the dimension $l_d$ of the $d^{th}$-graded part of the Lie algebra $\LL^u_r(M)$. We use the generating series to compute this from the universal enveloping algebra $H_*(\Omega M)$ as in \cite{rh4}.

%The generating series for $H_*(\Omega M)$ is $p(t)= \frac{1}{1-rt^{n-1} + t^{2n-2}}$. From the Poincar\'e-Birkhoff-Witt theorem the symmetric algebra on $\LL^u_r(M)$ is $H_*(\Omega M)$. Hence we have the equation 
%$$\frac{1}{\prod_d (1-t^d)^{l_d}} =   \frac{1}{1-rt^{n-1} + t^{2n-2}}$$
%Take log of both sides :
%\begin{eqnarray*}
%\log (1-rt^{n-1}+t^{2n-2} ) & = & \sum_d l_d\log (1-t^d)\\
%& = & -\sum_d l_d\left(t^d+\frac{t^{2d}}{2}+\frac{t^{3d}}{3}+\cdots\right)
%\end{eqnarray*}
%Expanding this and equating coefficients, we see that 
%$$
%\lambda_m:=\textup{coefficient of $t^m$ in $\log (1-rt^{n-1}+t^{2n-2})$} = -\frac{1}{m}\bigg(\sum_{d|m} d l_d \bigg).
%$$
%We use the M\"{o}bius inversion formula; it gives us 

%$$l_m =-\sum_{d|m} \mu(d)\frac{\lambda_{m/d}}{d}.  $$
%In fact, by expanding $\log (1-rt^{n-1}+t^{2n-2})$, we see that $\lambda_m =0$ if $m$ is not divisible by $n-1$ and 
%$$\lambda_{d(n-1)}= -\sum_{a+2b=d}(-1)^b{a+b \choose b}\frac{r^a}{a+b}$$
%We conclude that $l_m$ is $0$ if $m$ is not divisible by $n-1$ and 
%$$l_{d(n-1)} = \sum_{e|d}\frac{ \mu(e)}{d}\sum_{a+2b=\frac{d}{e}}(-1)^b{a+b \choose b}\frac{r^a}{a+b}$$

%\end{proof}

Recall that simply connected, finite cell complexes either have finite dimensional rational homotopy groups or exponential growth of ranks of rational homotopy groups (cf. \cite{FHT01}, \S 33). The former are called rationally elliptic while the latter are called rationally hyperbolic. From \cite{BerBor15} we note that the $(n-1)$-connected manifolds of dimension at most $(3n-2)$ with $H^*(M)$ having rank at least $4$ are all rationally hyperbolic. One may also verify this directly. Since the rank of $H^*(M)$ is at least $4$ the number of generating $u_i$ is at least $3$. Then one observes that after switching the ordering appropriately the word
$$u_1u_2u_1u_2u_1u_3$$
is a Lyndon word in degree $>2d$ as each $u_i$ has degree $>\frac{d}{3}$. So these manifolds cannot be rationally elliptic. This forces by the Milnor-Moore Theorem, that $\LL^{gr}(M)\otimes \Q$ has infinite rank. Hence, $\LL^u(M)$ also has infinite rank. 

There are many conjectures that lie in the dichotomy between rationally elliptic and hyperbolic spaces. We verify such a conjecture by Moore (\cite{FHT01}, pp. 518) below. For a rationally hyperbolic space $X$ the Moore conjecture states that there are primes $p$ for which the homotopy groups do not have any exponent at $p$, that is, for any power $p^r$ there is an element $\alpha\in \pi_*(X)$ of order $p^r$. We verify the following version.
\begin{theorem}\label{Moorehcm}
If $p\notin \Gamma$, the homotopy groups of $M$ do not have any exponent at $p$. 
\end{theorem}   

\begin{proof}
We have noted above that $\LL^u(M)$ has infinite rank. Thus there are elements of the Lyndon basis of arbitrarily large degree. Hence for arbitrarily large $l$, $\pi_*S^l$ occurs as a summand of $\pi_*M$. The proof is complete by observing  that any $p^s$ may occur as the order of an element in $\pi_*S^l$ for arbitrarily large $l$. This follows from the fact that the same is true for the stable homotopy groups and these can be realized as $\pi_k^s\cong \pi_{k+l}S^l$  for $l>k+1$. Now torsion of order $p^s$ for any $s$ occurs in the image of the $J$-homomorphism (cf. \cite{Rav86}, Theorem 1.1.13).  
\end{proof}

\subsection{The low rank cases}\label{r12}
We end by demonstrating the above computations when the rank of $H_*(M;\Q)$ is at most $4$. Since $H^0$ and $H^d$ are always $\Q$, we have to consider three possibilities: rank $2,3,4$. Our main techniques involve determining the rational homotopy type of $M$, of dimension $d$, and using it to compute the homotopy type at all but finitely many primes. 

In the rank $2$ case we know that rationally $M$ is a sphere  so that $M_\Q \simeq S^d_\Q$. Let $\Sigma$ denote the finite set of primes which occur as torsion in the homology of $M$. With $R_\Sigma$-coefficients, the $R_\Sigma$ localization $M_\Sigma$ of $M$ is a homology $d$-sphere. Thus $H^*(M;R)\cong H^*(S^d;R)$ for any ring $R$ lying between $R_\Sigma$ and $\Q$. Let $\alpha$ stand for the common notation for a generator of $H_*(M;R)$ for any such $R$. As $M_\Q \simeq S^d_\Q$, $\alpha$ lies in the image of the rational Hurewicz homomorphism. It follows that with $R_\Sigma$-coefficients, there is an integer $k$ so that $k\alpha$ lies in the image of the Hurewicz homomorphism. Let $\Gamma$ denote the union of $\Sigma$ and the prime divisors of $k$. Then $\Gamma$ is a finite set and $\alpha$ lies in the image of the $R_\Gamma$ Hurewicz homomorphism. It is now clear that there is a map $S^d_\Gamma \to M$ which is an isomorphism with  $R_\Gamma$-coefficients. Therefore, we have a homotopy equivalence  
$M_\Gamma \simeq S^d_\Gamma.$ Note that the torsion in the homology can be quite varied. So this is precisely the sort of result we are looking for. 

For the next case, let $J_2S^n$ be the second stage of the James construction which is obtained as the mapping cone of the Whitehead product $[id,id]$. If $H^*(M;\Q)$ has rank $3$, by Poincar\'e duality the cohomology ring is forced to be $\Q[x_s]/(x_s^3)$ where $d=2s$. By graded commutativity $s$ is forced to be even. The rational homotopy of such a space may be computed directly from the cohomology ring structure as the ring structure forces the space to be formal. The minimal model is given by $\Lambda(x_s,y_{3s-1})$ with $d(x_s)=0$ and $d(y_{3s-1})=x_s^3$. Thus the rational homotopy groups of $M$ are given by 
$$\pi_k^\Q(M) = \left\{\begin{array}{rl}
\Q  & \text{if $k=s,3s-1$} \\ 
0  &  \text{otherwise}.
\end{array}\right.$$   
It follows that on rationalizations we have map $S^s_\Q \to M_\Q$ which is an isomorphism on $\pi_*^\Q$ for $*\leq 2s-2$. In degree $2s-1$ the homotopy groups are $\pi_{2s-1}^\Q S^s\cong \Q\{[id,id]\}$ and $\pi^\Q_{2s-1}(M)=0$. Therefore the composite 
$$S^{2s-1}_\Q \xrightarrow{[id,id]} S^s_\Q \lra M_\Q$$
is null-homotopic and thus factors through the cofibre $(J_2S^s)_\Q$. Therefore we obtain a map $(J_2S^s)_\Q \to M_\Q$ which is an isomorphism on $H^s$ and by cup products also on $H^d$. As a result, one obtains that $M_\Q \simeq J_2S^{s}_\Q$. 

Next we upgrade the rational homotopy result to one which is valid after inverting finitely many primes, that is, over a set $\Gamma$ of finitely many primes that the $\Gamma$-localizations of the above two spaces are weakly equivalent. Let $\Sigma$ denote all the primes which appear as torsion in the homology of $M$. As $M$ is a compact CW-complex, the homotopy groups of $M$ are finitely generated. Let $\Gamma$ denote the primes in $\Sigma$ together with the finite list of primes which appear as torsion in $\pi_{2s-1}(M)$ and those which need to be inverted so that $x_s$ lies in the image of the Hurewicz homomorphism. We consider the following commutative diagram:
$$\xymatrix{S^{2s-1}_\Gamma \ar[r]^-{[id,id]} \ar[d] & S^s_\Gamma \ar[d] \ar[r]  & M_\Gamma \ar[d] \\ 
                      S^{2s-1}_\Q  \ar[r]^-{[id,id]}                   & S^s_\Q  \ar[r]                    & M_\Q                      }$$
The composite in the bottom row is $0$. The composite in the top row gives an element in $\pi_{2s-1}(M) \otimes R_\Gamma$ which injects into $\pi_{2s-1}(M)\otimes \Q$ by our choice of $\Gamma$. The latter group is $0$ from our choices. It follows that the composite of the top row is $0$ and thus we obtain a map from the mapping cone $J_2S^{s}_\Gamma \to M_\Gamma$ which is an isomorphism in cohomology with $R_\Gamma$ coefficients by the same argument as that for $\Q$. Thus we deduce that 
$$M_\Gamma \simeq (J_2S^{s})_\Gamma.$$

It remains to consider the last case when total rank is $4$. Let $\#^2J_2(n)$ denote the mapping cone of 
$$[id^1,id^1]+[id^2,id^2]: S^{2n-1}\to S^n\vee S^n$$
Then, 
$$H^*(\#^2J_2(n);\Q) \cong \Q[x_n,y_n]/(x_n^3,y_n^3,x_n y_n,x_n^2=y_n^2).$$   
If $H^*(M;\Q)$ has rank $4$, then by Poincar\'e duality the rational cohomology ring is forced to be one of the following: \\
(a) $\{ 1, x_s,y_s, x_s^2=y_s^2\}$ (where $d=2s$ with $s$ even), \\
(b) $\{ 1,x_k,y_{d-k},x_k \cdot y_{d-k} \} $. \\
Notice that (a) is the rational cohomology ring of $\#^2J_2(s)$ while (b) is the rational cohomology ring of $S^k \times S^{d-k}$. We now deduce that the rational homotopy type of $M$ must indeed be one of these. 

The rational homotopy groups may be computed directly as the ring structure forces the space to be formal. The minimal model for type (a) is given by $$\Lambda(x_s,y_s, u_{2s-1}, v_{2s-1}), d(x_s)=d(y_s)=0, d(u_{2s-1})=x_s^2-y_s^2, d(v_{2s-1})=x_sy_s.$$ 

Thus the rational homotopy groups of $M$ are given by 
$$\pi_k^\Q(M) = \left\{\begin{array}{rl}
\Q & \text{if $k=0$}\\
\Q^2 & \text{if $k=s, 2s-1$}\\ 
0 & \text{otherwise}.
\end{array}\right.$$   
Therefore, on rationalizations we have a map $S^s_\Q\vee S^s_\Q \to M_\Q$ representing $x_s$ and $y_s$ which is an isomorphism on $\pi_*^\Q$ for $*\leq 2s-2$. In degree $2s-1$ the homotopy group $$\pi_{2s-1}^\Q (S^s\vee S^s) \cong \Q\{[id^1,id^1],[id^1,id^2],[id^2,id^2]\}.$$ In the computation of homotopy groups using minimal models one knows that the quadratic part of the differential represents the Whitehead product, and so it follows that the element $[id^1,id^1]+[id^2,id^2]$ goes to $0$ in $M$.  Therefore the composite 
$$S^{2s-1}_\Q \xrightarrow{[id^1,id^1]+[id^2,id^2]} S^s_\Q \lra M_\Q$$
is null-homotopic and thus factors through the cofibre $(\#^2J_2(s))_\Q$. Therefore we obtain a map $(\#^2J_2(s))_\Q \to M_\Q$ which is an isomorphism on $H^s$ and by cup products also on $H^d$. It follows that $M_\Q \simeq \#^2J_2(s)_\Q$. 

If the cohomology algebra is of type (b), the minimal model matches that for the product $S^k\times S^{d-k}$. Thus, on rationalizations we have a map $S^k_\Q\vee S^{d-k}_\Q \to M_\Q$ which is an isomorphism on $\pi_*^\Q$ for $*\leq 2s-2$. In degree $2s-1$  the class $[id_k,id_{d-k}]$ generates a copy of $\Q$ in $\pi_{2s-1}^\Q (S^k\vee S^{d-k})$. As in the argument above, one shows that  $[id_k,id_{d-k}]$ goes to $0$ in $M$. Therefore we obtain a map $(S^k\times S^{d-k})_\Q \to M_\Q$ which is an isomorphism on $H^{\leq d-1}$ and by cup products also on $H^d$. It follows that $M_\Q \simeq (S^k\times S^{d-k})_\Q$. 

As in the $r=3$ case we upgrade the rational homotopy result to one which is valid after inverting finitely many primes. Let $\Sigma$ denote all the primes which appear as torsion in the homology of $M$ and $\Gamma$ denote the primes in $\Sigma$ together with the finite primes which appear as torsion in $\pi_{2s-1}(M)$, and so that the generators of $H_{\leq d-1}$ lie in the image of the $R_\Gamma$-Hurewicz homomorphism. Let $\phi$ denote $[id^1,id^1]+[id^2,id^2]$ or $[id^1,id^2]$ accordingly as $H^*(M)$ is of type (a) or (b). We consider the following commutative diagram 
$$\xymatrix{S^{2s-1}_\Gamma \ar[r]^{\phi} \ar[d] & S^s_\Gamma \vee S^s_\Gamma \ar[d] \ar[r]  & M_\Gamma \ar[d] \\ 
                      S^{2s-1}_\Q  \ar[r]^{\phi}                   & S^s_\Q \vee S^s_\Q                    \ar[r]            & M_\Q                      }$$
The composite in the bottom row is $0$. The composite in the top row gives an element in $\pi_{2s-1}(M) \otimes R_\Gamma$ which injects into $\pi_{2s-1}(M)\otimes \Q$ by our choice of $\Gamma$. The class $\phi$ maps to $0$ in the latter group as proved above. It follows that the composite of the top row is $0$ and thus we obtain a map from the mapping cone 
$$\mathit{Cone}(\phi)_\Gamma \to M_\Gamma$$
which is an isomorphism in cohomology with $R_\Gamma$ coefficients by the same argument as that for $\Q$. Thus we deduce that 
$$M_\Gamma \simeq \mathit{Cone}(\phi)_\Gamma.$$

\noindent We summarize all the above computations and observations in the result below.

\begin{theorem}\label{lowrank}
Let $M$ be a $(n-1)$-connected $d$-manifold with $d\leq 3n-2$. Suppose that the total rank of $H^*(M;\Q)$ is at most $4$. \\
(i) If the rank is $2$, then there is a finite set of primes $\Gamma$ such that $M_\Gamma \simeq S^d_\Gamma$.\\
(ii) If the rank is $3$, then there is a finite set of primes  $\Gamma$ such that $M_\Gamma \simeq J_2S^{ d/2}_\Gamma$.\\
(iii) If the rank is $4$, then there is a finite set of primes  $\Gamma$ such that $M_\Gamma \simeq (\#^2J_2(\frac{d}{2}))_\Gamma$ or  $M_\Gamma \simeq (S^k\times S^{d-k})_\Gamma$.\\
\end{theorem}

\bibliographystyle{siam}
\bibliography{htpy3n-2_ver2a.bbl}
\vspace*{0.5cm}

\end{document}